\documentclass{aic}

\aicAUTHORdetails{
  title = {The Maximum Number of Triangles in a Graph of Given Maximum Degree},
  author = {Zachary Chase},
  plaintextauthor = {Zachary Chase},
}  

\aicEDITORdetails{%
   year={2020},
   number={10},
   received={13 March 2020},   
   published={4 September 2020},  
   doi={10.19086/aic.16788},      
}   

\newcommand{\vs}{\vspace{3mm}}
\newtheorem{theorem}{Theorem}
\newtheorem*{conjecture*}{Conjecture}

\newtheorem{lemma}{Lemma}

\theoremstyle{remark}

\newtheorem*{remark*}{Remark}
\usepackage{tikz}
\tikzstyle{vertex}=[circle,draw=black,fill=black,inner sep=0,minimum size=3pt,text=white,font=\footnotesize]

\begin{document}

\begin{frontmatter}[classification=text]

\title{The Maximum Number of Triangles in a Graph of Given Maximum Degree} 

\author[zach]{Zachary Chase\thanks{The author is partially supported by Ben Green's Simons Investigator Grant 376201 and gratefully acknowledges the support of the Simons Foundation.}}

\begin{abstract}
We prove that any graph on $n$ vertices with max degree $d$ has at most $q{d+1 \choose 3}+{r \choose 3}$ triangles, where $n = q(d+1)+r$, $0 \le r \le d$. This resolves a conjecture of Gan-Loh-Sudakov.
\end{abstract}
\end{frontmatter}

\section{Introduction}
\indent \indent Fix positive integers $d$ and $n$ with $d+1 \le n \le 2d+1$. Galvin [7] conjectured that the maximum number of cliques in an $n$-vertex graph with maximum degree $d$ comes from a disjoint union $K_{d+1}\sqcup K_r$ of a clique on $d+1$ vertices and a clique on $r := n-d-1$ vertices. Cutler and Radcliffe [4] proved this conjecture. Engbers and Galvin [6] then conjectured that, for any fixed $t \ge 3$, the same graph $K_{d+1}\sqcup K_r$ maximizes the number of cliques of size $t$, over all $(d+1+r)$-vertex graphs with maximum degree $d$. Engbers and Galvin [6]; Alexander, Cutler, and Mink [1]; Law and McDiarmid [11]; and Alexander and Mink [2] all made progress on this conjecture before Gan, Loh, and Sudakov [9] resolved it in the affirmative. Gan, Loh, and Sudakov then extended the conjecture to arbitrary $n \ge 1$ (for any $d$). 

\begin{conjecture*}[Gan-Loh-Sudakov Conjecture] Any graph on $n$ vertices with maximum degree $d$ has at most $q{d+1 \choose 3}+{r \choose 3}$ triangles, where $n = q(d+1)+r$, $0 \le r \le d$.
\end{conjecture*}

They showed their conjecture implies that, for any fixed $t \ge 4$, any max-degree $d$ graph on $n = q(d+1)+r$ vertices has at most $q{d+1 \choose t}+{r \choose t}$ cliques of size $t$. In other words, considering triangles is enough to resolve the general problem of cliques of any fixed size. 

\vs

The Gan-Loh-Sudakov conjecture (GLS conjecture) has attracted substantial attention. Cutler and Radcliffe [5] proved the conjecture for $d \le 6$ and showed that a minimal counterexample, in terms of number of vertices, must have $q = O(d)$. Gan [8] proved the conjecture if $d+1-\frac{9}{4096}d \le r \le d$ (there are some errors in his proof, but they can be mended). Using fourier analysis, the author [3] proved the conjecture for Cayley graphs with $q \ge 7$. Kirsch and Radcliffe [10] investigated a variant of the GLS conjecture in which the number of edges is fixed instead of the number of vertices (with still a maximum degree condition). 

\vs

In this paper, we fully resolve the Gan-Loh-Sudakov conjecture. 

\begin{theorem}
For any positive integers $n,d \ge 1$, any graph on $n$ vertices with maximum degree $d$ has at most $q{d+1 \choose 3}+{r \choose 3}$ triangles, where $n = q(d+1)+r$, $0 \le r \le d$.
\end{theorem}

Analyzing the proof shows that $qK_{d+1}\sqcup K_r$ is the unique extremal graph if $r \ge 3$, and that $qK_{d+1}\sqcup H$, for any $H$ on $r$ vertices, are the extremal graphs if $0 \le r \le 2$. 

\vs

The heart of the proof is the following Lemma, of independent interest, which says that, in any graph, we can find a closed neighborhood whose removal from the graph removes few triangles. Theorem 1 will follow from its repeated application.

\begin{lemma}
In any graph $G$, there is a vertex $v$ whose closed neighborhood meets at most ${d(v)+1 \choose 3}$ triangles.
\end{lemma}

\vspace{-1mm}

As mentioned above, Theorem 1, together with the work of Gan, Loh, and Sudakov [9], yields the general result, for cliques of any fixed size.

\begin{theorem}
Fix $t \ge 3$. For any positive integers $n,d \ge 1$, any graph on $n$ vertices with maximum degree $d$ has at most $q{d+1 \choose t}+{r \choose t}$ cliques of size $t$, where $n = q(d+1)+r$, $0 \le r \le d$. 
\end{theorem}

Theorem 2 gives another proof of (the generalization of) Galvin's conjecture (to $n \ge 2d+2$) that a disjoint union of cliques maximizes the total number of cliques in a graph with prescribed number of vertices and maximum degree. 

\vs

Finally, the author would like to point out a connection to a related problem, that of determining the minimum number of triangles that a graph of fixed number of vertices $n$ and prescribed minimum degree $\delta$ can have. The connection stems from a relation, observed in [2] and [9], between the number of triangles in a graph and the number of triangles in its complement: $$|T(G)|+|T(G^c)| = {n \choose 3}-\frac{1}{2}\sum_v d(v)[n-1-d(v)].$$ Lo [12] resolved this ``dual" problem when $\delta \le \frac{4n}{5}$. His results resolve the GLS conjecture for regular graphs for $q=2,3$, and the GLS conjecture implies his results, up to an additive factor of $O(\delta^2)$, for $q=2,3$, and yields an extension of his results for $q \ge 4$ --- these are the optimal results asymptotically, in the natural regime of $\frac{\delta}{n}$ fixed, and $n \to \infty$. 

\section{Notation}

Denote by $E$ the edge set of $G$; for two vertices $u,v$, we write ``$uv \in E$" if there is an edge between $u$ and $v$ and ``$uv \not \in E$" otherwise --- in particular, for any $u$, $uu \not \in E$. For a vertex $v$, let $|T_{N[v]}|$ denote the number of triangles with at least one vertex in the closed neighborhood $N[v] := \{u : uv \in E\}\cup\{v\}$, and let $|T(G-N[v])|$ denote the number of triangles with all vertices in the graph $G-N[v]$ (the subgraph induced by the vertices not in $N[v]$). Finally, $d(v)$ denotes the degree of $v$. 

\section{Proof of Theorem 1}

\noindent For a graph $G$, let $W(G) = \{(x,u,v,w) : ux,vx,wx \in E, uv,uw,vw \not \in E\}$.

\begin{lemma}
For any graph $G$, $6\sum_v |T_{N[v]}| + |W(G)| = \sum_v d(v)^3$. 
\end{lemma}

\begin{proof}
Let $\Omega = \{(z,u,v,w) : uv,uw,vw \in E \text{ and } [zu \in E \text{ or } zv \in E \text{ or } zw \in E]\}$, $\Sigma = \{(x,u,v,w) : ux,vx,wx \in E\}$, and $W = W(G)$. Note that repeated vertices in the $4$-tuples are allowed. First observe that, since there are $6$ ways to order the vertices of a triangle, $\sum_v 6|T_{N[v]}| = |\Omega|$. Any 4-tuple in $\Sigma,W,$ or $\Omega$ gives rise to one of the induced subgraphs shown below, since one vertex must be adjacent to all the others. 

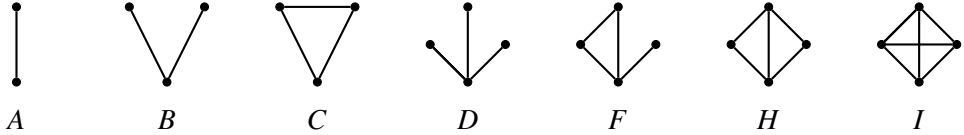
\vspace{4.5mm}

\hspace{5mm} \begin{tikzpicture}


\node[vertex] (A1) at (0,0) {};
\node[vertex] (A2) at (0,1) {};
\draw[thick] (A1) -- (A2);

\node[vertex] (B1) at (2,0) {};
\node[vertex] (B2) at (2.5,1) {};
\node[vertex] (B3) at (1.5,1) {};
\draw[thick] (B2) -- (B1) -- (B3);

\node[vertex] (C1) at (4,0) {};
\node[vertex] (C2) at (4.5,1) {};
\node[vertex] (C3) at (3.5,1) {};
\draw[thick] (C2) -- (C1) -- (C3) -- (C2);

\node[vertex] (D1) at (6,0) {};
\node[vertex] (D2) at (6.5,.5) {};
\node[vertex] (D3) at (5.5,.5) {};
\node[vertex] (D4) at (6,1) {};
\draw[thick] (D2) -- (D1) -- (D3) -- (D1) -- (D4);

\node[vertex] (E1) at (8,0) {};
\node[vertex] (E2) at (8.5,.5) {};
\node[vertex] (E3) at (7.5,.5) {};
\node[vertex] (E4) at (8,1) {};
\draw[thick] (E2) -- (E1) -- (E3) -- (E4) -- (E1);

\node[vertex] (F1) at (10,0) {};
\node[vertex] (F2) at (10.5,.5) {};
\node[vertex] (F3) at (9.5,.5) {};
\node[vertex] (F4) at (10,1) {};
\draw[thick] (F1) -- (F2) -- (F4) -- (F3) -- (F1) -- (F4);

\node[vertex] (G1) at (12,0) {};
\node[vertex] (G2) at (12.5,.5) {};
\node[vertex] (G3) at (11.5,.5) {};
\node[vertex] (G4) at (12,1) {};
\draw[thick] (G1) -- (G2) -- (G3) -- (G4) -- (G1) -- (G3) -- (G4) -- (G2);


\node at (0,-.5) {$A$};
\node at (2,-.5) {$B$};
\node at (4,-.5) {$C$};
\node at (6,-.5) {$D$};
\node at (8,-.5) {$F$};
\node at (10,-.5) {$H$};
\node at (12,-.5) {$I$};

\end{tikzpicture}

\vspace{1.5mm}

Since $|\Sigma| = \sum_v d(v)^3$, it thus suffices to show that for each of the induced subgraphs above, the number of times it comes from a 4-tuple in $\Sigma$ is the sum of the number of times it comes from 4-tuples in $\Omega$ and $W$. Any fixed copy of $A$, say on vertices $u$ and $v$, comes $0$ times from a 4-tuple in $\Omega$ (since it has no triangles), and $2$ times from each of $W$ and $\Sigma$ ($(u,v,v,v),(v,u,u,u)$). Any fixed copy of $B$, say on vertices $u,v,w$ with $vu,vw \in E$, comes $0$ times from $\Omega$, and $6$ times from each of $W$ and $\Sigma$ ($(v,u,u,w),(v,u,w,u),(v,u,w,w),(v,w,u,u),(v,w,u,w),(v,w,w,u)$). Any fixed copy of $C$ comes $18$ times from each of $\Omega$ and $\Sigma$ (3 choices for the first vertex and then $6$ for the ordered triangle), and $0$ times from $W$. Similarly, any fixed copy of $D$ comes 6 times from each of $W$ and $\Sigma$, and $0$ times from $\Omega$; finally, $F,H,I$ come $6,12,24$ times, respectively, from each of $\Omega$ and $\Sigma$, and $0$ times from $W$. 
\end{proof}

We now prove our key lemma, previously mentioned in the introduction.

\setcounter{lemma}{0}
\begin{lemma}
In any graph $G$, there is a vertex $v$ whose closed neighborhood meets at most ${d(v)+1 \choose 3}$ triangles, i.e. $|T_{N[v]}| \le {d(v)+1 \choose 3}$. 
\end{lemma}

\begin{proof}
By Lemma 2, since $|W(G)| \ge |\{(x,u,u,u) : ux \in E\}| = \sum_x d(x)$, we have $\sum_v |T_{N[v]}| \le \sum_v \frac{1}{6}[d(v)^3-d(v)]$. By the pigeonhole principle, there is some $v$ with $$|T_{N[v]}| \le \frac{1}{6}[d(v)^3-d(v)] = {d(v)+1 \choose 3}.$$
\end{proof}

\setcounter{lemma}{2}
\begin{lemma}
For any positive integers $a \ge b \ge 1$, it holds that ${a \choose 3}+{b \choose 3} \le {a+1 \choose 3}+{b-1 \choose 3}$. Consequently, for any positive integers $a,b$ and any positive integer $c$ with $\max(a,b) \le c \le a+b$, it holds that ${a \choose 3}+{b \choose 3} \le {c \choose 3}+{a+b-c \choose 3}$.
\end{lemma}

\begin{proof}
${a+1 \choose 3}-{a \choose 3} = {a \choose 2}$, and ${b \choose 3}-{b-1 \choose 3} = {b-1 \choose 2}$. Iterate to get the consequence.
\end{proof}

\vs

\noindent We now finish the proof of Theorem 1. 

\begin{proof}[Proof of Theorem 1]
With a fixed $d$, we induct on $n$. For $n=1$, the result is obvious. Take some $n \ge 2$, and suppose the theorem holds for all smaller values of $n$. Let $G$ be a max-degree $d$ graph on $n$ vertices. By Lemma $1$, we may take $v$ with $|T_{N[v]}| \le {d(v)+1 \choose 3}$. Write $n = q(d+1)+r$ for $0 \le r \le d$. Note $|T(G)| = |T(G-N[v])|+|T_{N[v]}|$. Since $G-N[v]$ has maximum degree (at most) $d$, if $d(v)+1 \le r$, then induction and Lemma 3 give $$|T(G)| \le q{d+1 \choose 3}+{r-(d(v)+1) \choose 3}+{d(v)+1 \choose 3} \le q{d+1 \choose 3}+{r \choose 3},$$ and if $d(v)+1 > r$, then induction and Lemma 3 give $$|T(G)| \le (q-1){d+1 \choose 3}+{d+1+r-(d(v)+1) \choose 3}+{d(v)+1 \choose 3} \le q{d+1 \choose 3}+{r \choose 3}.$$ The maximum degree condition ensured $d+1+r-(d(v)+1) \ge 0$ and $d(v)+1 \le d+1$.
\end{proof}

\section*{Acknowledgments} 
I would like to thank Po-Shen Loh for telling me the Gan-Loh-Sudakov conjecture and my advisor Ben Green for encouragement. I also thank Daniel Korandi for a cleaner proof of Lemma 2 and for helpful suggestions on the paper's presentation.

\bibliographystyle{amsplain}


\begin{aicauthors}
\begin{authorinfo}[zach]
  Zachary Chase\\
  University of Oxford\\
  Oxford, United Kingdom\\
  zachary\imagedot{}chase\imageat{}maths\imagedot{}ox\imagedot{}uk\imagedot{}edu \\
  \url{http://people.maths.ox.ac.uk/~chase/}
\end{authorinfo}
\end{aicauthors}

\end{document}